\documentclass[11pt]{amsart}

\usepackage[T1]{fontenc}
\usepackage{lmodern}
\usepackage{microtype}
\usepackage{amsmath,amssymb,amsthm,mathtools}
\usepackage{enumitem}
\usepackage{geometry}
\usepackage{xcolor}
\usepackage{hyperref}
\usepackage{booktabs}
\usepackage{aliascnt}

\hypersetup{
  colorlinks=true,
  linkcolor=blue!55!black,
  citecolor=green!35!black,
  urlcolor=blue!60!black,
  pdftitle={A Resolution of Erd\H{o}s Problem 550 on Tree versus Complete Multipartite Ramsey Numbers},
  pdfauthor={Eric Li}
}

\newtheorem{theorem}{Theorem}[section]
\newaliascnt{lemma}{theorem}
\newtheorem{lemma}[lemma]{Lemma}
\aliascntresetthe{lemma}
\newaliascnt{proposition}{theorem}
\newtheorem{proposition}[proposition]{Proposition}
\aliascntresetthe{proposition}
\newaliascnt{corollary}{theorem}
\newtheorem{corollary}[corollary]{Corollary}
\aliascntresetthe{corollary}
\newaliascnt{claim}{theorem}

\aliascntresetthe{claim}
\theoremstyle{definition}
\newaliascnt{definition}{theorem}

\aliascntresetthe{definition}
\newaliascnt{remark}{theorem}

\aliascntresetthe{remark}

\usepackage[nameinlink,capitalise,noabbrev]{cleveref}

\newcommand{\N}{\mathbb N}
\newcommand{\eps}{\varepsilon}
\newcommand{\ind}{\mathbf 1}
\newcommand{\cC}{\mathcal C}
\newcommand{\cD}{\mathcal D}
\newcommand{\Gb}{G_{\mathrm b}}
\newcommand{\Gr}{G_{\mathrm r}}

\title[A resolution of Erd\H{o}s Problem 550]
{A Resolution of Erd\H{o}s Problem 550 on Tree versus Complete Multipartite Ramsey Numbers}
\author[E. Li]{Eric Li}
\date{June 23, 2026}
\thanks{\textit{Email addresses:}
\href{mailto:el593@cam.ac.uk}{el593@cam.ac.uk},
\href{mailto:contact@ericli.com}{contact@ericli.com}}

\makeatletter
\def\@setauthors{%
  \begingroup
  \def\thanks{\protect\thanks@warning}%
  \trivlist
  \centering\footnotesize \@topsep30\p@\relax
  \advance\@topsep by -\baselineskip
  \item\relax
  \author@andify\authors
  \def\\{\protect\linebreak}%
  \MakeUppercase{\authors}\par
  \vspace{6pt}%
  {\normalfont\normalsize Trinity College, University of Cambridge}%
  \endtrivlist
  \endgroup
}
\makeatother

\begin{document}

\begin{abstract}
We resolve Erd\H{o}s Problem~550, originally asked as question~(2) of
Erd\H{o}s, Faudree, Rousseau, and Schelp.  Precisely, for fixed
$k\ge2$ and $1\le m_1\le\cdots\le m_k$, we prove that, for every
sufficiently large $n$ and every $n$-vertex tree $T$,

\[
 R\!\left(T,K_{m_1,\ldots,m_k}\right)
 \le (k-1)\bigl(R(T,K_{m_1,m_2})-1\bigr)+m_1.
\]
The proof combines an off-Tur\'an tree-embedding theorem, proved by
regularity and whole-edge allocation, with a compactness theorem for
bounded-rank hypergraph obstructions.
The full and unconditional proof has been formally verified in Lean.
\end{abstract}


\maketitle

\section{Introduction}

This paper resolves Erd\H{o}s Problem~550, and the full and unconditional proof
has been formally verified in Lean.\footnote{A complete machine-checked
formalisation of the entire proof is available in \cite{Erdos550Lean}; see also
the formal verification statement before the references.}

For graphs $J$ and $L$, let $R(J,L)$ be the least $N$ such that every
red--blue coloring of $K_N$ contains a red copy of $J$ or a blue copy of
$L$.  We use the symmetry $R(J,L)=R(L,J)$.  The graph
$K_{m_1,\ldots,m_k}$ is the complete $k$-partite graph with part sizes
$m_1,\ldots,m_k$.

Chv\'atal proved that
\[
 R(T,K_k)=(k-1)(|T|-1)+1
\]
for every tree $T$~\cite{Chvatal1977}.  More generally, Burr's canonical
construction gives, for every connected $n$-vertex graph $J$ with $n\ge2$
and every
graph $L$,
\[
 R(J,L)\ge (\chi(L)-1)(n-1)+\sigma(L),
\]
where $\sigma(L)$ is the smallest color-class size in a proper
$\chi(L)$-coloring of $L$~\cite{Burr1981}.  In particular, if
\[
 F=K_{m_1,\ldots,m_k},\qquad q=k-1,\qquad a=m_1,
\]
then
\begin{equation}\label{eq:burrlower}
 R(T,F)\ge q(n-1)+a.
\end{equation}
When $n>a$, the usual construction on $q(n-1)+a-1$ vertices consists of
$q$ red cliques of order $n-1$ and one red clique of order $a-1$, with all
crossing edges blue.  It contains no red $T$.  It also contains no blue
$F$: the host parts induce a proper $(q+1)$-coloring of any blue copy, and
in every proper $(q+1)$-coloring of the complete $(q+1)$-partite graph $F$,
each original part of $F$ is one color class, because the $q+1$ original
parts are nonempty and vertices in different parts are adjacent.  Thus some class of order at
least $a$ would have to lie in the host part of order $a-1$.  When
$2\le n\le a$, instead color every edge blue.  There is no red $T$, since
every tree of order at least two has an edge, while
\[
 q(n-1)+a-1\le(q+1)(a-1)<|F|,
\]
so there is no blue $F$ either.  This proves \Cref{eq:burrlower} in all
cases with $n\ge2$; the first case is the one relevant to the asymptotic
theorem.  For completeness, $R(K_1,F)=1$, so the displayed Burr bound is not
asserted when $n=1$.

Erd\H{o}s, Faudree, Rousseau, and Schelp proved the corresponding large-tree
result when the smallest part has order $1$, and asked whether, for arbitrary
fixed part sizes,
\begin{equation}\label{eq:EFQ}
 R\!\left(T,K_{m_1,\ldots,m_k}\right)
 \le (k-1)\bigl(R(T,K_{m_1,m_2})-1\bigr)+m_1
\end{equation}
for all sufficiently large trees; this is question~(2) of
\cite[p.~153]{EFRS1989} and is recorded as Erd\H{o}s Problem~550
\cite{ErdosProblems550}.  The main theorem below resolves this problem.

\begin{theorem}[Resolution of Erd\H{o}s Problem~550]\label{thm:main}
Fix an integer $k\ge2$ and integers
\[
 1\le m_1\le\cdots\le m_k.
\]
There exists $n_0=n_0(m_1,\ldots,m_k)$ such that, for every $n\ge n_0$ and
every $n$-vertex tree $T$,
\[
 R\!\left(T,K_{m_1,\ldots,m_k}\right)
 \le (k-1)\bigl(R(T,K_{m_1,m_2})-1\bigr)+m_1.
\]
\end{theorem}

Writing $r=R(T,K_{m_1,m_2})$, Burr's lower construction applied to
$K_{m_1,m_2}$ gives
\[
 r\ge(n-1)+m_1\ge n.
\]
Together with \Cref{eq:burrlower,thm:main}, this gives the quantitative
interpretation
\begin{equation}\label{eq:excess}
 0\le R(T,F)-\bigl(q(n-1)+a\bigr)\le q(r-n).
\end{equation}
Thus the excess over the canonical Ramsey-goodness lower bound is controlled
by the excess in the bipartite problem involving the two smallest classes.

The proof follows the chain
\[
 \begin{gathered}
 \text{uniform EFRS asymptotic}
 \Longrightarrow \text{off-Tur\'an embedding}\\
 \Longrightarrow \text{near-Tur\'an red density}
 \Longrightarrow \text{stable reservoirs}\\
 \Longrightarrow \text{profile and blocker inequalities}
 \Longrightarrow \text{compactness rounding}\\
 \Longrightarrow \text{Ramsey capacity contradiction}.
 \end{gathered}
\]
The uniform asymptotic, regularity~\cite{Szemeredi1978}, stability, and the
K\H{o}v\'ari--S\'os--Tur\'an theorem are standard inputs.  The
local rooted regular-pair lemma of Hladk\'y and Piguet is used after a
whole-edge allocation and a stateful packedness argument.  The off-Tur\'an
theorem and the null-blocker compactness theorem are proved here.

For context, bounded-degree tree Ramsey goodness was developed by Balla,
Pokrovskiy, and Sudakov~\cite{BallaPokrovskiySudakov2018}, and the optimal
fixed-chromatic-number form for bounded-degree trees versus general fixed
graphs was subsequently obtained by Montgomery, Pavez-Sign\'e, and
Yan~\cite{MontgomeryPavezSigneYan2025}.  These results do not cover arbitrary
trees.  Cooperative coloring was introduced for families of graphs in
\cite{AharoniEtAl2015} and has also been studied for hypergraphs in
\cite{BaiEtAl2024}.  The recent work of Mi and Wang
\cite{MiWang2026} concerns a different regime in which one part of the
complete multipartite graph grows.

\section{Notation and external inputs}\label{sec:inputs}

All graphs are finite and simple, and all graph copies are ordinary,
non-induced copies.  If $J$ is a graph and $A,B\subseteq V(J)$ are disjoint,
then $e_J(A,B)$ denotes the number of edges with one endpoint in $A$ and one
in $B$.  Write $e_J(A)$ for the number of edges of $J[A]$.  For a
vertex $x$, write $d_J(x,A)=|N_J(x)\cap A|$.  Write
$\overline d(J)=2e(J)/|V(J)|$ for average degree, and write $\mathbf 1_E$
for the indicator of an event $E$.  A set is independent in a hypergraph if
it contains no hyperedge.  A cooperative coloring of hypergraphs
$(\mathcal H_1,\ldots,\mathcal H_q)$ on a common vertex set is a partition
$X_1\dot\cup\cdots\dot\cup X_q$ such that $X_i$ is independent in
$\mathcal H_i$ for every $i$.  For $q\ge1$, let $t_q(N)$ denote the number
of edges in the balanced $q$-partite Tur\'an graph on $N$ vertices, so
\[
 t_q(N)=\frac{q-1}{2q}N^2+O_q(1).
\]
Throughout, $[q]=\{1,\ldots,q\}$ and $y_+=\max\{y,0\}$ for real $y$.
Hypergraph edges are always nonempty unless explicitly stated otherwise.  We
set the rank of the empty hypergraph to $0$; otherwise, rank is the maximum
edge order.  We use the convention
\[
 \bigcap_{x\in\varnothing}A(x)=\Omega
\]
for events $A(x)\subseteq\Omega$.

\begin{proposition}[EFRS uniform asymptotic]\label{prop:efrs}
Let $Q$ be a fixed graph without isolated vertices and with chromatic number
$s\ge2$.  For every $\theta>0$ there is $n_0=n_0(Q,\theta)$ such that, for
every $n\ge n_0$ and every $n$-vertex tree $T$,
\[
 \bigl|R(T,Q)-(s-1)n\bigr|\le\theta n.
\]
In particular, for fixed positive integers $a,b$,
\[
 R(T,K_{a,b})=n+o(n)
\]
uniformly over all $n$-vertex trees.
\end{proposition}

\begin{proof}
The uniform corollary of Erd\H{o}s, Faudree, Rousseau, and
Schelp~\cite[Corollary, p.~316]{EFRS1985} gives the upper estimate
\[
 R(Q,T)\le (s-1+\theta)n
\]
after increasing its uniform threshold if necessary.  Burr's standard
lower construction gives
\[
 R(Q,T)=R(T,Q)\ge (s-1)(n-1)+\sigma(Q).
\]
After another harmless increase of the threshold, these two estimates and
Ramsey symmetry give the displayed two-sided estimate.
\end{proof}

\begin{proposition}[Erd\H{o}s--Simonovits stability]\label{prop:stability}
Fix a graph $F$ with $\chi(F)=q+1$.  Let $(J_N)$ be a sequence of $F$-free
graphs, where $|V(J_N)|=N$ and $N\to\infty$.  If
\[
 e(J_N)\ge t_q(N)-o(N^2),
\]
then $V(J_N)$ has a partition into $q$ parts for which the total number of
edges internal to the parts is $o(N^2)$.
\end{proposition}

This is the Erd\H{o}s--Simonovits stability theorem in the form used below
\cite{Simonovits1968}.

We also use the K\H{o}v\'ari--S\'os--Tur\'an theorem in the form
\[
 \operatorname{ex}(N,K_{a,b})=o(N^2)
\]
for fixed $a,b$~\cite{KST1954}.  When $a=1$, the assertion needed below is
the elementary fact that a $K_{1,b}$-free graph has maximum degree at most
$b-1$.

\section{A tailored off-Tur\'an tree-embedding theorem}\label{sec:offturan}

The next theorem is the only density-to-embedding statement needed in the
Ramsey proof.  Its complement hypothesis is used at the reduced-graph level,
so no average-degree form of the Erd\H{o}s--S\'os conjecture is required.

We use the regularity lemma with an initial prepartition from the published
article of Hladk\'y and Piguet~\cite[Theorem~5.11]{HladkyPiguet2016}.  Its
per-cluster irregular-pair bound implies the following global consequence:
an $(\varepsilon,\ell)$-regular partition
\[
 V(G)=V_0\mathbin{\dot\cup}V_1\mathbin{\dot\cup}\cdots
      \mathbin{\dot\cup}V_\ell
\]
has equal nonexceptional cluster order $s$, has
$|V_0|<\varepsilon|V(G)|$, and has at most $\varepsilon\ell^2$
irregular unordered pairs of nonexceptional clusters.

For disjoint nonempty vertex sets $P,Q$, write
\[
 d(P,Q)=\frac{e(P,Q)}{|P||Q|}.
\]
The pair $(P,Q)$ is \emph{$\varepsilon$-regular} if
\[
 |d(P',Q')-d(P,Q)|\le\varepsilon
\]
whenever $P'\subseteq P$ and $Q'\subseteq Q$ satisfy the strict inequalities
$|P'|>\varepsilon|P|$ and $|Q'|>\varepsilon|Q|$.  Such subsets are called
\emph{significant}.  Given significant $Q'\subseteq Q$, a vertex $v\in P$
is \emph{typical to $Q'$} if
\[
 |N(v)\cap Q'|\ge\bigl(d(P,Q)-\varepsilon\bigr)|Q'|;
\]
the symmetric definition applies to vertices of $Q$.

We first record the local fact used in the stateful matching embedding.

\begin{lemma}[Rooted regular-pair embedding]\label{lem:rootedpair}
Let $(P,Q)$ be an $\varepsilon$-regular pair of equal order $s$ and density
at least $d$, where $d>2\varepsilon>0$.  Let $(K,r)$ be a rooted tree and
let $P'\subseteq P_0\subseteq P$ and $Q'\subseteq Q_0\subseteq Q$.  Let
$h$ be a positive integer.  If
\[
 \min\{|P_0|,|Q_0|\}\ge h,
 \qquad
 \max\{|P'|,|Q'|\}\ge h,
 \qquad
 h\ge\frac{\varepsilon s+|K|}{d-2\varepsilon},
\]
then $K$ embeds in $P_0\cup Q_0$, with $r$ in $P'\cup Q'$.  If
$|P'|\ge h$, the root may be put in $P'$ and its bipartition class in
$P_0$; the analogous assertion holds for $Q'$.
\end{lemma}

\begin{proof}
The root-placement statement is Hladk\'y--Piguet's
Lemma~5.12~\cite{HladkyPiguet2016}; the bipartition-side refinement follows
from the following standard greedy proof.  Choose sets
$B_P\subseteq P_0$ and $B_Q\subseteq Q_0$, each of order $h$.  Since a pair
density is at most $1$, the numerical hypothesis gives
\[
 h\ge\frac{\varepsilon s+|K|}{d-2\varepsilon}>\varepsilon s.
\]
Thus $B_P$ and $B_Q$ are significant sets in the strict regularity
convention.  At most $\varepsilon s$ vertices on either side fail to be
typical to the selected set on the other side.  If $|P'|\ge h$, choose the
root image in $P'$ typical to $B_Q$; this image need not belong to $B_P$.
Embed every other vertex in the root's bipartition class in $B_P$, and
embed the opposite class in $B_Q$.  The case $|Q'|\ge h$ is symmetric.

For completeness, embed vertices in parent-before-child order, maintaining
that every embedded nonroot vertex is typical to the selected set on the side
into which its children will be embedded.  Suppose that $u$ vertices have
already been embedded and that the next embedded parent has $c$
as-yet-unembedded children.  Its
typicality gives at least $(d-\varepsilon)h$ neighbors in the opposite
selected set.  After excluding the at most $\varepsilon s$ vertices that
are atypical to the selected set on the parent's side and the $u$ used
vertices, the number of permissible neighbors is at least
\[
 (d-\varepsilon)h-\varepsilon s-u
 \ge |K|-u+\varepsilon h\ge c.
\]
Here we used $(d-2\varepsilon)h\ge\varepsilon s+|K|$ and
$u+c\le |K|$.  Choosing the children typical to the other selected set
maintains the induction.  The two tree bipartition classes remain on their
prescribed sides.
\end{proof}

\begin{corollary}[Fine-partition consequences]\label{cor:fineconsequences}
Let $T$ be a rooted $n$-vertex tree and let $k$ be an integer with
$1\le k<n-1$.  There is a $k$-fine partition
$(S_X,S_Y,\mathcal D_X,\mathcal D_Y)$ satisfying
\[
 |S_X|,|S_Y|\le\frac{12(n-1)}{k}
\]
such that the root belongs to $S_X\cup S_Y$ and
$\mathcal D_X\mathbin{\dot\cup}\mathcal D_Y$ is precisely the set of
components of $T-(S_X\cup S_Y)$.  Every such component $K$ has at most $k$
vertices.  Put $S=S_X\cup S_Y$, and call the members of
\[
 \partial_S K=N_T(V(K))\cap S
\]
the \emph{boundary seeds} of $K$.  For $Z\in\{X,Y\}$ and
$K\in\mathcal D_Z$,
\[
 \varnothing\ne\partial_S K\subseteq S_Z.
\]
We say in this case that $K$ is \emph{routed to $Z$}.  Distances within
each $S_Z$ are even and distances between $S_X$ and $S_Y$ are odd; in
particular, every edge of $T[S]$ joins $S_X$ to $S_Y$.  Moreover, all
vertices of $K$ adjacent to $\partial_S K$ lie in one bipartition class of
$K$.
\end{corollary}

\begin{proof}
Apply Definition~5.2 and Lemma~5.3 of Hladk\'y and
Piguet~\cite{HladkyPiguet2016} with parameter $k$, and write the resulting
fine partition as
$(W_X,W_Y,\mathcal D_X,\mathcal D_Y)$.  In the construction proving that
result, the shrubs are precisely the components left after the cut set
$W_X\cup W_Y$ is removed.  Because the root is in the cut set and $T$ is
connected, every remaining component has at least one cut-set neighbor.
The construction gives the asserted component partition, the order bound,
inclusion of the root, and the bounds on $|W_X|$ and $|W_Y|$; it also puts
every boundary seed of a shrub in its corresponding $W_Z$ and makes all
distances within each $W_Z$ even and all distances from $W_X$ to $W_Y$ odd.

It remains only to record the parity consequence used below.  Let $xu$ and
$x'u'$ be two boundary incidences of $K$.  The seeds $x,x'$ lie in the same
$W_Z$, so their distance in $T$ is even.  If $x\ne x'$, the unique
$x$--$x'$ path in $T$ passes through $K$, and therefore
\[
 \operatorname{dist}_K(u,u')=\operatorname{dist}_T(x,x')-2,
\]
which is even.  If $x=x'$, then $u=u'$: otherwise the edges $xu,xu'$ and
the $u$--$u'$ path in $K$ would form a cycle in $T$.  Thus in all cases
$u,u'$ lie in the same bipartition class of $K$.  Take
$S_X=W_X$ and $S_Y=W_Y$.
\end{proof}

\begin{lemma}[Whole-edge allocation]\label{lem:wholeedgeallocation}
Let $J$ be a finite set, let $s,t>0$, and let
$w_X,w_Y:J\to[0,2s]$.  Put
\[
 W_X=\sum_{j\in J}w_X(j),\qquad
 W_Y=\sum_{j\in J}w_Y(j).
\]
Let $R_X,R_Y,C_0\ge0$, put $a_Z=R_Z+C_0$ for $Z\in\{X,Y\}$, and suppose
\[
 W_X>0,\qquad W_Y>0,
 \qquad
 \frac{a_X+t}{W_X}+\frac{a_Y+t}{W_Y}<1,
 \qquad 2|J|s^2<t^2.
\]
Then $J$ has a partition $J=J_X\mathbin{\dot\cup}J_Y$ such that
\[
 \sum_{j\in J_Z}w_Z(j)\ge R_Z+C_0
 \qquad(Z=X,Y).
\]
\end{lemma}

\begin{proof}
Choose
\[
 \frac{a_X+t}{W_X}\le\lambda\le
 1-\frac{a_Y+t}{W_Y}.
\]
Let $(I_j)_{j\in J}$ be independent Bernoulli variables with
$\Pr(I_j=1)=\lambda$, and define
\[
 S_X=\sum_{j\in J}I_jw_X(j),\qquad
 S_Y=\sum_{j\in J}(1-I_j)w_Y(j).
\]
Then $\mathbb ES_X=\lambda W_X\ge a_X+t$ and
$\mathbb ES_Y=(1-\lambda)W_Y\ge a_Y+t$.  Since
$\lambda(1-\lambda)\le1/4$ and $w_Z(j)\le2s$,
\[
 \begin{aligned}
 \operatorname{Var}S_X+\operatorname{Var}S_Y
 &=\sum_{j\in J}\lambda(1-\lambda)
       \bigl(w_X(j)^2+w_Y(j)^2\bigr)\\
 &\le2|J|s^2<t^2.
 \end{aligned}
\]
Thus the expectation of
$(S_X-\mathbb ES_X)^2+(S_Y-\mathbb ES_Y)^2$ is less than $t^2$.
For some outcome both absolute deviations are less than $t$.  Taking
$J_X=\{j:I_j=1\}$ and $J_Y=J\setminus J_X$ for that outcome gives
$S_Z\ge\mathbb ES_Z-t\ge a_Z$, as required.
\end{proof}

\begin{lemma}[One-edge state]\label{lem:oneedgestate}
Let $(C,D)$ be an $\varepsilon$-regular pair of order $s$ and density at
least $d>2\varepsilon$.  Let $K$ be a rooted tree, let $h$ be a positive
integer, and let $\tau,e_0,\mu\ge0$ satisfy
\[
 |K|\le\tau,\qquad
 \varepsilon s+|K|\le(d-2\varepsilon)h,
 \qquad \mu=h+\tau+2e_0.
\]
Let $U_C\subseteq C$ and $U_D\subseteq D$ be used sets, with
$c=|U_C|$ and $d'=|U_D|$.  Let $0\le L,R\le s$, let
$A_C\subseteq C$ and $A_D\subseteq D$, and let
\[
 P_C\subseteq A_C\setminus U_C,\qquad
 P_D\subseteq A_D\setminus U_D
\]
satisfy
\[
 |P_C|\ge L-c-e_0,\qquad |P_D|\ge R-d'-e_0.
\]
Suppose
\begin{equation}\label{eq:oneedge-surplus}
 c+d'+h+\mu\le L+R
\end{equation}
and the packedness alternative
\begin{equation}\label{eq:oneedge-packed}
 \min\{L,R\}-\mu\le\min\{c,d'\}
 \quad\text{or}\quad |c-d'|\le\tau
\end{equation}
holds.  Then $K$ embeds in $(C\setminus U_C)\cup(D\setminus U_D)$ so that
its root lies in $P_C\cup P_D$ and its entire root bipartition class lies
in the corresponding one of $A_C\setminus U_C,A_D\setminus U_D$.
After adding the embedded vertices to the used sets, the alternative
\Cref{eq:oneedge-packed} still holds with $c,d'$ replaced by the two new
loads.
\end{lemma}

\begin{proof}
First suppose the first alternative in \Cref{eq:oneedge-packed} holds.
By \Cref{eq:oneedge-surplus},
\[
 (L-c-e_0)+(R-d'-e_0)
 \ge h+\mu-2e_0=2h+\tau\ge2h,
\]
so at least one of $P_C,P_D$ has order at least $h$.  Suppose without loss
of generality that $L\le R$.  Packedness gives
$c,d'\ge L-\mu$, and \Cref{eq:oneedge-surplus} then gives
\[
 c+h\le L+R-d'-\mu\le R\le s,
 \qquad
 d'+h\le L+R-c-\mu\le R\le s.
\]
Thus both full unused sides have order at least $h$.

Now suppose the first alternative fails.  The second holds.  Put
$m=\min\{c,d'\}$ and $M=\max\{c,d'\}$.  Then
$m<\min\{L,R\}-\mu$ and $M\le m+\tau$.  On either side,
\[
 \text{threshold}-\text{load}-e_0
 >\mu-\tau-e_0=h+e_0\ge h.
\]
Both root pools and both full unused sides therefore have order at least
$h$.  If the bipartition classes of $K$ have orders $a_K,b_K$, choose their
orientation so that the new load difference is at most $\tau$.  Indeed, the
two possible differences are
$(c-d')\pm(a_K-b_K)$, and the smaller absolute value is at most
\[
 \max\{|c-d'|,|a_K-b_K|\}\le\tau.
\]

In the first case orient the root toward a pool of order at least $h$; in
the second use the balanced orientation just chosen.  Apply
\Cref{lem:rootedpair} with the root-side available set equal to the unused
part of the corresponding $A_C$ or $A_D$, the other available set equal to
the full unused opposite side, and the root set equal to the chosen pool.
The numerical hypothesis holds by assumption.  This gives the asserted
embedding and root-class containment.  In the first case the first
packedness alternative persists because loads only increase; in the second
case the chosen orientation preserves the second alternative.
\end{proof}

\begin{lemma}[Seed--shrub scheduling]\label{lem:seedshrubschedule}
Let $T$ be a rooted tree, let $S\subseteq V(T)$ contain its root, and let
$\mathcal D$ be the set of components of $T-S$.  Contract every member of
$\mathcal D$ to a single shrub-node while retaining the vertices of $S$.
The resulting graph is a rooted tree.  Every shrub-node has exactly one
incident edge toward the root; its endpoint in $S$ is the upper seed of the
shrub, and every other boundary seed is a child of the shrub-node.

Consequently, the following parent-before-child procedure terminates after
embedding all of $T$: start with the root seed; embed a shrub when its upper
seed has been embedded; after embedding that shrub, expose all of its lower
seeds; and embed a seed when its parent seed or parent shrub vertex has been
embedded.  If each local step uses fresh vertices and preserves every
required contact for the newly exposed lower seeds, the resulting map is an
injective embedding of $T$.
\end{lemma}

\begin{proof}
Contracting pairwise disjoint connected subgraphs of a tree preserves
connectedness and creates neither a cycle nor a loop.  Nor can it create
parallel edges: two edges from the same seed to one shrub, together with
the path inside the shrub, would form a cycle in $T$.  The quotient is
therefore a tree.  Its orientation toward the retained root gives every
nonroot vertex, including every shrub-node, a unique parent edge.  This is
the asserted upper incidence; all other incidences at that shrub-node point
to child seeds.  Induction on distance from the quotient root now proves
that the procedure exposes every quotient vertex.  The parent-before-child
embeddings inside each shrub then expose precisely the corresponding lower
seeds.  Freshness gives injectivity, and the stated contact condition
realizes every edge between a shrub and a lower seed.
\end{proof}

\begin{theorem}[Off-Tur\'an embedding]\label{thm:offturan}
Fix an integer $q\ge2$ and positive integers
\[
 1\le m_1\le\cdots\le m_{q+1}.
\]
Put
\[
 a=m_1,\qquad b=m_2,\qquad
 H=K_{a,b},\qquad F=K_{m_1,\ldots,m_{q+1}}.
\]
For every $\delta>0$ there exists
$n_0=n_0(q,m_1,\ldots,m_{q+1},\delta)$ such that the following holds.
Let $T$ be an $n$-vertex tree with $n\ge n_0$, put
\[
 r=R(T,H),\qquad N=q(r-1)+a,
\]
and let $\Gb$ be a graph on $N$ vertices.  If
\[
 F\nsubseteq\overline{\Gb}
\]
and
\begin{equation}\label{eq:offhyp}
 e(\Gb)\ge\binom N2-t_q(N)+\delta N^2,
\end{equation}
then $T\subseteq\Gb$.  The threshold is uniform over all $n$-vertex trees.
\end{theorem}

\begin{proof}
Write $f=e(F)$.  By \Cref{prop:efrs}, uniformly in $T$,
\begin{equation}\label{eq:offscale}
 r=n+o(n),\qquad N=qn+o(n),\qquad \frac Nq=n+o(N).
\end{equation}
Throughout this proof, every $o(N)$ term is uniform over the choice of the
$n$-vertex tree $T$.  This follows from \Cref{prop:efrs} and from the fact
that every regularity parameter below depends only on $q,F,\delta$.
Since
\[
 \frac{2}{N}\left(\binom N2-t_q(N)\right)
 =\frac Nq-1+O_q(N^{-1}),
\]
\Cref{eq:offhyp} gives
\begin{equation}\label{eq:offav}
 \overline d(\Gb)\ge n+2\delta N+o(N).
\end{equation}

Choose $\eta>0$ so small that
\begin{equation}\label{eq:etachoice}
 \eta<\min\left\{
 \frac{\delta}{200},
 \frac{q-1}{200q},
 \frac{1}{100(f+1)},
 10^{-4}
 \right\}.
\end{equation}
Choose $m_0$ so large that
\begin{equation}\label{eq:m0choice}
 \frac2{m_0}<\eta,
 \qquad \frac1{m_0}<\eta^2,
 \qquad \eta m_0>4q.
\end{equation}
Now choose
\begin{equation}\label{eq:epschoice}
 0<\varepsilon<\min\left\{
 \frac{\eta^3}{400},
 \frac{\eta^2}{8q},
 \frac{\eta}{100}
 \right\}.
\end{equation}
By \Cref{eq:offav}, for all sufficiently large $n$,
\begin{equation}\label{eq:offbigav}
 \overline d(\Gb)\ge n+200\eta N,
 \qquad
 n+80\eta N<N.
\end{equation}

Apply Hladk\'y--Piguet's regularity lemma
\cite[Theorem~5.11]{HladkyPiguet2016} with the trivial one-class
prepartition and the already fixed parameters $\varepsilon,m_0$.  It gives
an $(\varepsilon,\ell)$-regular partition
\[
 V(\Gb)=V_0\mathbin{\dot\cup}V_1\mathbin{\dot\cup}\cdots
       \mathbin{\dot\cup}V_\ell,
\]
where
\[
 m_0\le\ell\le L,
 \qquad |V_0|<\varepsilon N,
 \qquad |V_1|=\cdots=|V_\ell|=s.
\]
Let $\Gb^*$ be obtained by deleting the edges incident with $V_0$, the
edges internal to clusters, the edges in irregular pairs, and the edges in
regular pairs of blue density less than $\eta$.  The four deletion bounds
are
\[
 \varepsilon N^2,
 \qquad
 \ell\binom{s}{2}\le\frac{N^2}{2m_0},
 \qquad
 \varepsilon\ell^2s^2\le\varepsilon N^2,
 \qquad
 \eta\binom{\ell}{2}s^2\le\frac{\eta N^2}{2}.
\]
Thus
\[
 \overline d(\Gb)-\overline d(\Gb^*)
 \le\left(4\varepsilon+\frac1{m_0}+\eta\right)N<2\eta N.
\]
With $\mathcal P=\{V_1,\ldots,V_\ell\}$, put
\[
 D(C)=\frac{e_{\Gb^*}(C,V(\Gb))}{s}.
\]
Every edge of $\Gb^*$ joins two clusters, and therefore
\begin{equation}\label{eq:Bstarav}
 (n+100\eta N)\ell\le\sum_{C\in\mathcal P}D(C),
 \qquad
 \sum_{C\in\mathcal P}D(C)=\frac{2e(\Gb^*)}{s}.
\end{equation}
Indeed, the average of the $D(C)$ is at least
$2e(\Gb^*)/N\ge n+198\eta N$.  Also $0\le D(C)\le N$.

Let $Q$ be the graph on $\mathcal P$ whose edges are the regular pairs of
blue density at least $\eta$.  We claim
\begin{equation}\label{eq:alphaQ}
 \alpha(Q)<\eta\ell.
\end{equation}
If an independent set $I$ has order $p\ge\eta\ell$, then
$p>4q$ by \Cref{eq:m0choice}.  The graph of regular pairs on $I$ has at
least $\binom p2-\varepsilon\ell^2$ edges, while
\[
 \binom p2-t_q(p)
 \ge\frac{p^2}{2q}-\frac p2
 \ge\frac{p^2}{4q}
 \ge\frac{\eta^2\ell^2}{4q}
 >\varepsilon\ell^2.
\]
It contains a $(q+1)$-clique.  All its pairs have blue density less than
$\eta$, because $I$ is independent in $Q$.  Independently for the $q+1$
clusters, choose a uniformly random $m_i$-subset of the $i$th cluster (or,
equivalently, a uniformly random ordered $m_i$-tuple of distinct vertices).
For each required cross-edge, its two selected endpoints are marginally
uniform in the corresponding pair of clusters, so the probability that the
edge is blue is less than $\eta$.  A union bound over the $f$ required
cross-edges gives probability less than $f\eta<1$ of seeing any blue
required edge.  The cluster order tends to infinity because $\ell\le L$,
so the prescribed fixed subset orders are available.  A choice with no
blue required edge therefore gives a red $F$, a contradiction.

Let
\[
 \mathcal A=\{C\in\mathcal P:D(C)\ge n+80\eta N\},
 \qquad \theta=\frac{|\mathcal A|}{\ell}.
\]
By \Cref{eq:Bstarav},
\[
 n+100\eta N
 \le\theta N+(1-\theta)(n+80\eta N),
\]
and hence $\theta\ge20\eta$.  Thus
$|\mathcal A|>\alpha(Q)$, so $Q[\mathcal A]$ has an edge $XY$.

Take a maximal matching
\[
 \mathcal M=\{C_jD_j:j\in J\}
\]
in $Q-\{X,Y\}$, and let $U$ be the union of its endpoint clusters.  The
vertices of $Q-\{X,Y\}$ left unmatched by $\mathcal M$ form an independent
set, so fewer than $\eta\ell$ of them are unmatched.
Consequently
\[
 |(V_1\cup\cdots\cup V_\ell)\setminus U|
 \le(\eta\ell+2)s
 \le\eta N+\frac{2N}{m_0}<2\eta N.
\]
For $Z\in\{X,Y\}$ and $C\in\mathcal P$, define
\[
 p_Z(C)=
 \begin{cases}
  e_{\Gb}(Z,C)/s^2,&ZC\in E(Q),\\
  0,&ZC\notin E(Q),
 \end{cases}
\]
and put
\begin{equation}\label{eq:wholeweights}
 w_Z(j)=s\bigl(p_Z(C_j)+p_Z(D_j)\bigr).
\end{equation}
Removing the clusters outside $U$ from the normalized degree of either
heavy head costs less than $2\eta N$.  Hence
\begin{equation}\label{eq:headdegrees}
 \sum_{j\in J}w_X(j)\ge n+78\eta N,
 \qquad
 \sum_{j\in J}w_Y(j)\ge n+78\eta N.
\end{equation}

Set
\[
 \tau_0=\frac{\eta^2}{128L},
 \qquad
 k_T=\lfloor\tau_0n\rfloor.
\]
For all sufficiently large $n$, $1\le k_T<n-1$.  Root $T$, apply
\Cref{cor:fineconsequences} with $k=k_T$, and denote the resulting
$k_T$-fine partition by
\[
 (W_X,W_Y,\mathcal D_X,\mathcal D_Y).
\]
Put
\[
 S=W_X\cup W_Y,
 \qquad R_X=\sum_{K\in\mathcal D_X}|V(K)|,
 \qquad R_Y=\sum_{K\in\mathcal D_Y}|V(K)|.
\]
The members of $\mathcal D_X\cup\mathcal D_Y$ are the components of $T-S$;
we call them shrubs.  In particular, \Cref{cor:fineconsequences} says that
every shrub has order at most $k_T$ and has at least one boundary seed in
$S$, all its boundary seeds lie in the same one of $W_X,W_Y$, and all their
neighbors in the shrub lie in the same shrub bipartition class.  For all
sufficiently large $n$, $k_T\ge\tau_0n/2$, and hence
\begin{equation}\label{eq:seedbounds}
 |S|\le\frac{24(n-1)}{k_T}\le\frac{48}{\tau_0},
 \qquad R_X+R_Y=n-|S|\le n.
\end{equation}

Since $\ell s=N-|V_0|>(1-\varepsilon)N$, $\ell\le L$, and
$N/n\to q\ge2$, we have $s\ge n/(4L)$ for all sufficiently large $n$.
Consequently
\[
 k_T\le\tau_0n=\frac{\eta^2n}{128L}\le\frac{\eta^2s}{32}.
\]
Define
\begin{equation}\label{eq:localparameters}
 \tau=\frac{\eta^2s}{32},
 \qquad h=\left\lceil\frac{\eta s}{8}\right\rceil,
 \qquad e_0=\varepsilon s,
 \qquad \mu=h+\tau+2e_0.
\end{equation}
Indeed, $\varepsilon<\eta^3/400$ and
$\varepsilon<\eta/100$, while $h\ge\eta s/8$, so
\[
 \varepsilon s+\tau
 <\left(\frac{\eta^3}{400}+\frac{\eta^2}{32}\right)s
 <\frac{49\eta^2}{400}s
 <(\eta-2\varepsilon)\frac{\eta s}{8}
 \le(\eta-2\varepsilon)h.
\]
Thus
\begin{equation}\label{eq:state-num2}
 \varepsilon s+\tau\le(\eta-2\varepsilon)h.
\end{equation}

We now construct the fixed head cores and the one-sided contact sets.  Let
$\mathcal T$ be the set of the $2|J|$ matching endpoint clusters and put
\[
 \Theta=\frac{8\varepsilon\ell}{\eta}.
\]
For $Z\in\{X,Y\}$ and $z\in Z$, call an endpoint $C\in\mathcal T$ bad for
$z$ if
\[
 d_{\Gb}(z,C)<(p_Z(C)-\varepsilon)s.
\]
For each endpoint, regularity shows that at most $\varepsilon s$ vertices
of $Z$ call it bad; endpoints with $p_Z(C)=0$ are never bad.  Therefore the
sum of the bad counts over $z\in Z$ is at most
$|\mathcal T|\varepsilon s\le\ell\varepsilon s$.  Let $Z^\circ$ consist of
the vertices of $Z$ that have at least $(\eta-\varepsilon)s$ neighbors in
the other head cluster and have bad count at most $\Theta$.  Regularity of
$XY$ and Markov's inequality give
\begin{equation}\label{eq:headcoreloss}
 |Z\setminus Z^\circ|
 \le\varepsilon s+\frac{\ell\varepsilon s}{\Theta}
 =\left(\varepsilon+\frac\eta8\right)s.
\end{equation}
Consequently
\begin{equation}\label{eq:headcoresize}
 |Z^\circ|\ge\left(1-\varepsilon-\frac\eta8\right)s>\varepsilon s.
\end{equation}
By \Cref{eq:seedbounds}, for all sufficiently large $n$,
\begin{equation}\label{eq:seedroom}
 |S|<\varepsilon|Z^\circ|,
 \qquad |S|<|Z^\circ|,
\end{equation}
and every vertex in one head core has more than $|S|$ neighbors in the
other.  The last assertion follows by subtracting
\Cref{eq:headcoreloss} from its $(\eta-\varepsilon)s$ neighbors in the full
opposite cluster.

For an endpoint $C\in\mathcal T$, define the trimmed threshold
\begin{equation}\label{eq:trimthreshold}
 t_Z(C)=\max\{0,\,p_Z(C)s-2\varepsilon s\}.
\end{equation}
If $t_Z(C)>0$, let
\[
 R_Z(C)=\{v\in C:d_{\Gb}(v,Z^\circ)\ge(p_Z(C)-\varepsilon)|Z^\circ|\};
\]
if $t_Z(C)=0$, let $R_Z(C)=\varnothing$.  A positive threshold implies
$ZC\in E(Q)$ and $p_Z(C)>2\varepsilon$.  Since
$|Z^\circ|>\varepsilon s$ by \Cref{eq:headcoresize}, regularity gives
\begin{equation}\label{eq:retainedloss}
 |C\setminus R_Z(C)|\le\varepsilon s=e_0
 \quad\text{when }t_Z(C)>0.
\end{equation}
Every vertex of a nonempty $R_Z(C)$ has more than
$\varepsilon|Z^\circ|>|S|$ neighbors in $Z^\circ$.

It remains to assign whole matching edges to the two heads.  Put
\begin{equation}\label{eq:C0}
 C_0=2s\Theta+|J|(h+\mu)+2\varepsilon N.
\end{equation}
Here $2s\Theta$ pays for all edges bad for a chosen seed image,
$|J|(h+\mu)$ is one local reserve per edge, and $2\varepsilon N$ pays for
the two endpoint trimmings in \Cref{eq:trimthreshold}.  Since
$|J|\le\ell/2$, $\ell s\le N$, and $h\le\eta s/8+1$,
\begin{align*}
 C_0
 &\le \frac{16\varepsilon}{\eta}N
   +\frac\ell2\left(\frac{\eta s}{4}+2
       +\frac{\eta^2s}{32}+2\varepsilon s\right)
   +2\varepsilon N\\
 &\le \frac{16\varepsilon}{\eta}N
   +\frac{\eta N}{8}+\ell+\frac{\eta^2N}{64}
   +3\varepsilon N
 <3\eta N                                                     \tag{\(*\)}\label{eq:C0bound}
\end{align*}
for all sufficiently large $n$.

Let $a_X=R_X+C_0$, $a_Y=R_Y+C_0$, and $t=\eta N$.  By
\Cref{eq:seedbounds,eq:C0bound},
\[
 \begin{aligned}
 (a_X+t)+(a_Y+t)
 &=R_X+R_Y+2C_0+2\eta N\\
 &<n+8\eta N.
 \end{aligned}
\]
The two denominators in \Cref{eq:headdegrees} are at least
$n+78\eta N$.  Therefore
\begin{equation}\label{eq:ratio-room}
 \frac{a_X+t}{\sum_jw_X(j)}+
 \frac{a_Y+t}{\sum_jw_Y(j)}
 \le\frac{(a_X+t)+(a_Y+t)}{n+78\eta N}<1.
\end{equation}
Moreover,
\[
 2|J|s^2\le\ell s^2\le\frac{N^2}{m_0}<t^2.
\]
Apply \Cref{lem:wholeedgeallocation} with demands $R_X,R_Y$ and reserve
$C_0$.  It gives a partition $J=J_X\mathbin{\dot\cup}J_Y$ satisfying
\begin{equation}\label{eq:whole-edge-supply}
 R_Z+C_0\le\sum_{j\in J_Z}w_Z(j)
 \qquad(Z=X,Y).
\end{equation}

We finish with the stateful embedding.  Apply
\Cref{lem:seedshrubschedule} to $S$ and the shrubs.  This supplies the unique
upper boundary incidence of every shrub, the corresponding root of the
shrub, and the parent-before-child schedule.  A shrub is ready when its upper
seed has been embedded; its other boundary seeds are lower seeds.  Map
$W_X$ into $X^\circ$ and $W_Y$ into $Y^\circ$ as the seeds become ready.

For $j\in J_Z$, write
\[
 L_j=t_Z(C_j),\qquad R_j=t_Z(D_j),
\]
let $U_{C,j}\subseteq C_j$ and $U_{D,j}\subseteq D_j$ be the current used
sets, and put $u_j=|U_{C,j}|$ and $v_j=|U_{D,j}|$.
Maintain the packedness invariant
\begin{equation}\label{eq:packed}
 \min\{L_j,R_j\}-\mu\le\min\{u_j,v_j\}
 \quad\text{or}\quad
 |u_j-v_j|\le\tau.
\end{equation}
It holds initially by the second alternative.

Suppose a shrub $K$ routed to $Z$ is ready and its upper seed has image
$z\in Z^\circ$.  Let $J_Z(z)$ be the allocated edges for which neither endpoint
is bad for $z$.  Since the matching endpoints are distinct,
\[
 |J_Z\setminus J_Z(z)|\le\Theta.
\]
For a good endpoint, $z$ has at least its trimmed threshold many neighbors.
Moreover, trimming both endpoints of all allocated edges loses at most
$2\varepsilon N$, while deleting the bad edges loses at most $2s\Theta$.
Thus \Cref{eq:C0,eq:whole-edge-supply} give the explicit calculation
\begin{equation}\label{eq:static-surplus}
 \begin{aligned}
 \sum_{j\in J_Z(z)}(L_j+R_j)
 &\ge \sum_{j\in J_Z}w_Z(j)-2s\Theta-2\varepsilon N\\
 &\ge R_Z+|J|(h+\mu)\\
 &\ge R_Z+|J_Z(z)|(h+\mu).
 \end{aligned}
\end{equation}
All vertices already used in matching regions assigned to $Z$ belong to
shrubs routed to $Z$, so their total number is at most $R_Z$.  Hence
\Cref{eq:static-surplus} implies
\[
 \sum_{j\in J_Z(z)}(u_j+v_j)
   +|J_Z(z)|(h+\mu)
 \le\sum_{j\in J_Z(z)}(L_j+R_j).
\]
The ready shrub $K$ is routed to $Z$, so $R_Z\ge |K|\ge1$.  If $J_Z(z)$
were empty, \Cref{eq:static-surplus} would therefore read $R_Z\le0$.
Thus $J_Z(z)$ is nonempty, and some $j\in J_Z(z)$ satisfies
\begin{equation}\label{eq:local-surplus}
 u_j+v_j+h+\mu\le L_j+R_j.
\end{equation}

Inside the retained sets define the two unused root pools
\[
 P_C=(N_{\Gb}(z)\cap R_Z(C_j))\setminus U_{C,j},\qquad
 P_D=(N_{\Gb}(z)\cap R_Z(D_j))\setminus U_{D,j}.
\]
By goodness of the selected edge for $z$ and \Cref{eq:retainedloss}, their
respective orders are at least
\begin{equation}\label{eq:rootpools}
 L_j-u_j-e_0,
 \qquad R_j-v_j-e_0.
\end{equation}
For example, when $L_j>0$, goodness of $C_j$ for $z$ and
\Cref{eq:retainedloss} give
\[
 |P_C|\ge(p_Z(C_j)-\varepsilon)s-e_0-u_j
 =L_j-u_j\ge L_j-u_j-e_0;
\]
the calculation for $P_D$ is identical.
When a threshold is zero its retained set is empty and the corresponding
lower bound is nonpositive, so this assertion is also valid in that case.
Apply \Cref{lem:oneedgestate} to $(C_j,D_j)$ with density parameter
$d=\eta$, used sets $U_{C,j},U_{D,j}$, thresholds $L_j,R_j$, retained sets
$R_Z(C_j),R_Z(D_j)$, and root pools $P_C,P_D$.  Its size and regular-pair
hypotheses follow from $|K|\le k_T\le\tau$ and
\Cref{eq:state-num2}; its surplus and packedness hypotheses are
\Cref{eq:local-surplus,eq:packed}.  The lemma embeds $K$, preserves
\Cref{eq:packed} after updating the used sets and their loads, and places
the entire root bipartition class in the
corresponding retained set.  All lower
boundary seeds of $K$ lie in the same class of the bipartition of $T$ as
its upper boundary seed.  Consequently, by the final parity assertion of
\Cref{cor:fineconsequences}, their respective neighbors in $K$ lie in the
same bipartition class of $K$, namely the class containing the root of $K$.
Those neighbors therefore lie in the retained root class and retain
more than $|S|$ neighbors in the appropriate head core $Z^\circ$.

For a ready seed, there are three possibilities.  The root seed may be put in
any unused vertex of its head core.  A seed whose parent is a seed uses an
unused cross-head neighbor; the parent and child lie in opposite head cores
by \Cref{cor:fineconsequences}.  A seed whose parent lies in a shrub uses an
unused neighbor supplied by the retained contact set.  In every case
\Cref{eq:seedroom} and the fact that fewer than $|S|$ seeds have already been
embedded leave a choice.  Matching regions are pairwise disjoint and disjoint
from the two head clusters, so all extensions are injective.  The contact
condition required by \Cref{lem:seedshrubschedule} was verified in the
preceding paragraph.  That lemma now shows that the schedule terminates only
after every vertex has been embedded.  This gives $T\subseteq\Gb$.

Finally, the dependency order is
\[
 (q,F,\delta)\longrightarrow\eta\longrightarrow m_0
 \longrightarrow\varepsilon\longrightarrow L
 \longrightarrow\tau_0\longrightarrow n_0.
\]
Thus $L$ and $\tau_0$ are independent of $T$, while $k_T,h,\tau,\mu$ are
prescribed by the displayed formulas.  The opening uniformity convention
therefore permits one $n_0$ for all the large-$n$ requirements and all
$n$-vertex trees.
\end{proof}

\begin{corollary}[Near-Tur\'an red density]\label{cor:nearTuran}
With the notation of \Cref{thm:offturan}, suppose that $\Gb$ is $T$-free
and $\Gr:=\overline{\Gb}$ is $F$-free.  Then, uniformly over all
$n$-vertex trees,
\[
 e(\Gr)\ge t_q(N)-o(N^2).
\]
\end{corollary}

\begin{proof}
For every fixed $\delta>0$, \Cref{thm:offturan} implies, for all sufficiently
large $n$,
\[
 e(\Gb)<\binom N2-t_q(N)+\delta N^2.
\]
Thus $e(\Gr)>t_q(N)-\delta N^2$.  To make the uniform $o(N^2)$ term
explicit, let $n_{\mathrm{off}}(1/j)$ be the uniform threshold supplied by
\Cref{thm:offturan} with $\delta=1/j$, put $n_0=0$, and define recursively
\[
 n_j=\max\{n_{j-1}+1,n_{\mathrm{off}}(1/j)\}\qquad(j\ge1).
\]
Define $\delta_n=1$ for $n<n_1$ and
\[
 \delta_n=\frac1j
 \qquad\text{when }n_j\le n<n_{j+1}.
\]
Then $\delta_n\to0$, uniformly over the $n$-vertex tree $T$, and for every
$n\ge n_1$ the preceding inequality is exactly
\[
 e(\Gr)\ge t_q(N)-\delta_nN^2.
\]
\end{proof}

\section{Rooted forest embedding and reservoir profiles}\label{sec:treeprofile}

We first isolate the elementary arbitrary-tree embedding used after the
stability step.

\begin{lemma}[Prescribed-root forest embedding]\label{lem:rootedforest}
Let $\mathcal F$ be a forest on $t$ vertices, with one distinguished root in
each component.  If a graph $J$ satisfies $\delta(J)\ge t-1$, then every
injection of the roots of $\mathcal F$ into $V(J)$ extends to an embedding of
$\mathcal F$ in $J$.
\end{lemma}

\begin{proof}
Root every component at its distinguished vertex and order the nonroot
vertices so that every vertex appears after its parent.  Pre-embed the
roots.  Immediately before a nonroot vertex is embedded, at most $t-2$
already used vertices other than its parent image are forbidden.  The
parent image has at least $t-1$ neighbors, so an unused neighbor remains.
\end{proof}

\begin{lemma}[Count-and-load allocation]\label{lem:allocation}
Fix $q\ge2$ and $\omega>0$.  There are constants
$\kappa=\kappa(\omega)>0$ and $\delta_0=\delta_0(\omega)>0$ with the
following property.  Let $s_1,\ldots,s_d$ be positive integers satisfying
\[
 s_j\le \frac n2,
 \qquad
 \sum_{j=1}^d s_j=n-1,
\]
and let $c_1,\ldots,c_q$ be nonnegative integers satisfying
\[
 c_i\le(1+\delta_0)n
 \quad(i\in[q]),
 \qquad
 \sum_{i=1}^q c_i\ge(1+\omega)n.
\]
Then $[d]$ has a partition
$I_1\mathbin{\dot\cup}\cdots\mathbin{\dot\cup}I_q$ such that
\[
 |I_i|\le c_i
 \quad\text{and}\quad
 \sum_{j\in I_i}s_j\le(1-\kappa)n
 \qquad(i\in[q]).
\]
\end{lemma}

\begin{proof}
Set $\delta_0=\omega/4$ and choose
$0<\kappa<\min\{\omega/8,1/8\}$.  Since
$d\le n-1<\sum_i c_i$, some assignment satisfies the count constraints.
Among all such assignments, choose one minimizing the maximum load
$L_i=\sum_{j\in I_i}s_j$.

Suppose $L_h>(1-\kappa)n$.  This heavy bin is unique, since all other loads
together are less than $\kappa n$.  Hence fewer than $\kappa n$ components
are assigned outside $h$.  Moreover,
\[
 \sum_{i\ne h}c_i
 \ge(1+\omega)n-(1+\delta_0)n
 =\frac{3\omega n}{4}.
\]
Since $\kappa<\omega/8$, some bin $i\ne h$ has an unused count slot.  Move
any component from $h$ to $i$.  The new load of $i$ is less than
$\kappa n+n/2<(1-\kappa)n$, while the old unique maximum strictly
decreases, a contradiction.
\end{proof}

\begin{lemma}[Profile lemma]\label{lem:profile}
Fix $q\ge2$ and $\omega>0$.  There are $\gamma>0$ and $n_0$ such that the
following holds for every $n\ge n_0$.  Let $T$ be an $n$-vertex tree, and
suppose a blue graph contains pairwise disjoint sets $W_1,\ldots,W_q$ and a
vertex $x\notin\bigcup_iW_i$ satisfying
\[
 (1-\gamma)n\le |W_i|\le(1+\gamma)n,
 \qquad
 \delta(\Gb[W_i])\ge |W_i|-\gamma n
\]
for every $i\in[q]$.  If
\[
 \sum_{i=1}^q d_{\Gb}(x,W_i)\ge(1+\omega)n
\]
then $\Gb$ contains $T$.
\end{lemma}

\begin{proof}
Let $\kappa,\delta_0$ be supplied by \Cref{lem:allocation} for $\omega$, and
choose
\[
 0<\gamma<\min\{\delta_0,\kappa/2\}.
\]
It is enough to take $n_0=2$.  Choose a centroid $z$ of $T$.  The components
of $T-z$ have orders
$s_1,\ldots,s_d$ with $s_j\le n/2$ and $\sum_j s_j=n-1$.  Put
$c_i=d_{\Gb}(x,W_i)$.  Then $c_i\le |W_i|\le(1+\delta_0)n$, so
\Cref{lem:allocation}
assigns the components to the reservoirs so that at most $c_i$ components
and at most $(1-\kappa)n$ vertices are assigned to $W_i$.

Map $z$ to $x$.  For each component assigned to $W_i$, map its neighbor of
$z$ to a distinct blue neighbor of $x$ in $W_i$.  Put
\[
 t_i=\sum_{j\in I_i}s_j .
\]
Then $t_i\le(1-\kappa)n$, and
\[
 \delta(\Gb[W_i])\ge |W_i|-\gamma n
 \ge(1-2\gamma)n\ge(1-\kappa)n\ge t_i-1.
\]
\Cref{lem:rootedforest} embeds all assigned rooted forests, giving a blue
copy of $T$.
\end{proof}

\section{Null-blocker compactness}\label{sec:compactness}

We now prove the abstract rounding theorem used in the Ramsey argument.  The
intuition is as follows.  For each vertex $x$ and coordinate $i$, an event
$A_i(x)$ records the certificate that $x$ is compatible with color $i$.  A
random point in each probability space usually makes $x$ miss exactly one
coordinate; we then assign $x$ to that missing coordinate.  The $a$-set
null-intersection condition prevents many vertices from lying in all sampled
coordinates, while the signed expectation balances those vertices against the
ones missing two or more coordinates.  This is why the natural deletion bound is
$a-1$: the exceptional all-compatible set cannot contain an $a$-set.

\subsection{Exact countable rounding}

\begin{theorem}[Exact null-blocker rounding]\label{thm:exactrounding}
Fix integers $q\ge2$ and $a\ge1$.  Let $X$ be finite or countable.  For each
$i\in[q]$, let $(\Omega_i,\mu_i)$ be a probability space, let
$A_i(x)\subseteq\Omega_i$ be measurable for every $x\in X$, and let
$\cD_i$ be a hypergraph on $X$ whose edges are nonempty and finite.  Write
\[
 \rho_i(x)=\mu_i(A_i(x)).
\]
Assume the following.

\begin{enumerate}[label=\textup{(N\arabic*)},leftmargin=3em]
\item\label{it:N1}
For every $x\in X$,
\[
 \sum_{i=1}^q\rho_i(x)\ge q-1.
\]

\item\label{it:N2}
For every $S\in\binom Xa$, some $i\in[q]$ satisfies
\[
 \mu_i\left(\bigcap_{x\in S}A_i(x)\right)=0.
\]

\item\label{it:N3}
For every $i\in[q]$ and every $E\in\cD_i$, some $j\ne i$ satisfies
\[
 \mu_j\left(\bigcap_{x\in E}A_j(x)\right)=0.
\]
\end{enumerate}

Then there are $Z\subseteq X$, with $|Z|\le a-1$, and a partition
\[
 X\setminus Z=X_1\mathbin{\dot\cup}\cdots\mathbin{\dot\cup}X_q
\]
such that $X_i$ is independent in $\cD_i$ for every $i$.

Moreover, if
\[
 h(x)\in\operatorname*{arg\,min}_{i\in[q]}\rho_i(x),
\]
then the partition may be chosen so that $x\in X_{h(x)}$ for all but
finitely many $x\in X$.
\end{theorem}

\begin{proof}
Put
\[
 e(x)=\rho_{h(x)}(x)=\min_i\rho_i(x).
\]
On the product probability space
\[
 (\Omega,\mu)=\prod_{i=1}^q(\Omega_i,\mu_i),
\]
define
\[
 Q_x=A_1(x)\times\cdots\times A_q(x).
\]
If $|X|<a$, every point of $\Omega$ belongs to at most $a-1$ of the sets
$Q_x$.  Otherwise, \ref{it:N2} implies that every intersection of $a$
distinct sets $Q_x$ has measure zero.  The family of $a$-subsets of a
countable set is countable, so outside one null set every point belongs to
at most $a-1$ of the sets $Q_x$.  Tonelli's theorem yields
\begin{equation}\label{eq:productsum}
 \sum_{x\in X}\prod_{i=1}^q\rho_i(x)
 =\int_\Omega\sum_{x\in X}\ind_{Q_x}\,d\mu
 \le a-1.
\end{equation}

Condition \ref{it:N1} gives
\[
 \sum_{j\ne h(x)}(1-\rho_j(x))
 =(q-1)-\sum_{j\ne h(x)}\rho_j(x)
 \le e(x).
\]
Hence $\rho_j(x)\ge1-e(x)$ for $j\ne h(x)$.  If $e(x)\le1/2$, then
\[
 \prod_i\rho_i(x)
 \ge e(x)(1-e(x))^{q-1}
 \ge2^{-(q-1)}e(x).
\]
If $e(x)>1/2$, then every coordinate exceeds $1/2$, so the product exceeds
$2^{-q}$.  It follows from \Cref{eq:productsum} that only finitely many
vertices have $e(x)>1/2$ and that
\begin{equation}\label{eq:esummable}
 \sum_{x\in X}e(x)<\infty.
\end{equation}

Choose the coordinates $\omega_i\sim\mu_i$ independently and put
\[
 P_i=\{x\in X:\omega_i\in A_i(x)\},
 \qquad
 t(x)=|\{i:x\in P_i\}|.
\]
Call $x$ \emph{canonical} if
\[
 x\notin P_{h(x)}
 \quad\text{and}\quad
 x\in P_j\quad\text{for every }j\ne h(x).
\]
By the union bound,
\[
 \Pr(x\text{ is noncanonical})
 \le e(x)+\sum_{j\ne h(x)}(1-\rho_j(x))
 \le2e(x).
\]
The first Borel--Cantelli lemma, which does not require independence between
the events indexed by $x$, and \Cref{eq:esummable} imply that almost surely
only finitely many vertices are noncanonical.

Let
\[
 Y=\{x:t(x)=q\}.
\]
If $|X|<a$, then $|Y|\le|X|\le a-1$ trivially.  Otherwise, for every fixed
$S\in\binom Xa$, independence between the $q$ sampled coordinates and
\ref{it:N2} give
\[
 \Pr(S\subseteq Y)
 =\prod_{i=1}^q
 \mu_i\left(\bigcap_{x\in S}A_i(x)\right)=0.
\]
Taking a countable union over $S$, almost surely
\begin{equation}\label{eq:Ybound}
 |Y|\le a-1.
\end{equation}

For every edge $E\in\cD_i$, choose a blocker $j(E)\ne i$ supplied by
\ref{it:N3}.  The collection of finite subsets of a countable set is
countable, so the union of all edge sets of the $\cD_i$ is countable.
Consequently, almost surely
\begin{equation}\label{eq:blockpattern}
 E\nsubseteq P_{j(E)}
 \qquad\text{for every }i\text{ and }E\in\cD_i.
\end{equation}

Let
\[
 U=\{x:t(x)\le q-2\}.
\]
Both $U$ and $Y$ are almost surely finite because they consist of
noncanonical vertices.  Define
\[
 g_x=\ind_{\{t(x)\le q-2\}}-\ind_{\{t(x)=q\}}.
\]
Pointwise, $g_x\le q-1-t(x)$, and therefore
\[
 \mathbb E g_x
 \le q-1-\sum_i\rho_i(x)
 \le0.
\]
Moreover,
\[
 |g_x|\le\ind_{\{x\text{ is noncanonical}\}},
\]
so \Cref{eq:esummable} gives
\[
 \sum_x\mathbb E|g_x|<\infty.
\]
Thus the random series
\[
 G:=\sum_xg_x
\]
converges absolutely almost surely and in $L^1$.  Define $G=0$ on the
exceptional null set where the series does not converge.  Then $G$ is
integer-valued everywhere, and almost surely
\[
 G=|U|-|Y|.
\]
Moreover, $L^1$ convergence gives
\[
 \mathbb EG=\sum_x\mathbb Eg_x\le0.
\]
Since $G$ is integer-valued, $\Pr(G\le0)>0$; otherwise $G\ge1$ almost
surely.  The intersection of this positive-probability event with all the
preceding probability-one events still has positive probability.  Fix an
outcome in that intersection.  Then
\[
 |U|\le|Y|\le a-1.
\]

Delete $Z=U$.  Every remaining vertex has $t(x)=q-1$ or $t(x)=q$.  If
$t(x)=q-1$, assign $x$ to its unique missing coordinate; if $t(x)=q$,
assign it arbitrarily.  Let $X_i$ be the set assigned to coordinate $i$.
Every vertex of $X_i$ lies in $P_j$ for every $j\ne i$.  Thus an edge
$E\in\cD_i$ contained in $X_i$ would satisfy
$E\subseteq P_{j(E)}$, contradicting \Cref{eq:blockpattern}.  Hence $X_i$
is independent in $\cD_i$.

Every canonical vertex has unique missing coordinate $h(x)$.  Since only
finitely many vertices are noncanonical, all but finitely many vertices are
assigned to their homes.
\end{proof}

\subsection{Finite approximate systems}

\begin{theorem}[Null-blocker compactness]\label{thm:compactness}
Fix integers $q\ge2$, $a\ge1$, and $r_*\ge1$.  There exists
$\varepsilon_0=\varepsilon_0(q,a,r_*)>0$ such that the following holds.
Let $X$ be finite.  For each $i\in[q]$, let $(\Omega_i,\mu_i)$ be a finite
probability space, let $A_i(x)\subseteq\Omega_i$ be a measurable event for
$x\in X$, and let
$\cC_i$ be a hypergraph on $X$ whose edges are nonempty and whose rank is at
most $r_*$.  Write $\rho_i(x)=\mu_i(A_i(x))$.  Suppose that
$0\le\varepsilon\le\varepsilon_0$ and that
\begin{enumerate}[label=\textup{(A\arabic*)},leftmargin=3em]
\item\label{it:A1}
for every $x\in X$,
\[
 \sum_{i=1}^q\rho_i(x)\ge q-1-\varepsilon;
\]
\item\label{it:A2}
for every $S\in\binom Xa$,
\[
 \min_{i\in[q]}\mu_i\left(\bigcap_{x\in S}A_i(x)\right)\le\varepsilon;
\]
\item\label{it:A3}
for every $i\in[q]$ and every $E\in\cC_i$,
\[
 \min_{j\ne i}\mu_j\left(\bigcap_{x\in E}A_j(x)\right)\le\varepsilon.
\]
\end{enumerate}
Then there are $Z\subseteq X$, $|Z|\le a-1$, and a partition
\[
 X\setminus Z=X_1\mathbin{\dot\cup}\cdots\mathbin{\dot\cup}X_q
\]
such that $X_i$ is independent in $\cC_i$ for every $i$.
\end{theorem}

We prove the finite theorem by contradiction.  If no such $\varepsilon_0$
exists, then for every $\nu$ there is a counterexample with
$\varepsilon_\nu\le1/\nu$.  In such a sequence, an \emph{actual home} of
$x\in X_\nu$ is a coordinate
\[
 h_\nu(x)\in\operatorname*{arg\,min}_{i\in[q]}\rho_{i,\nu}(x),
\]
and we put $e_\nu(x)=\min_i\rho_{i,\nu}(x)$.

\begin{lemma}[Finite-dimensional limit]\label{lem:fdlimit}
Let $(X_\nu,A_{i,\nu},\cC_{i,\nu})$ be counterexamples with
$\varepsilon_\nu\to0$ and $|X_\nu|\to\infty$.  Label
\[
 X_\nu=\{x_{\nu,1},\ldots,x_{\nu,t_\nu}\}
\]
so that
\[
 e_\nu(x_{\nu,1})\ge e_\nu(x_{\nu,2})\ge\cdots.
\]
After passing to a subsequence, there are probability spaces
$(\Omega_i,\mu_i)$, measurable events $A_i(\ell)\subseteq\Omega_i$, and homes
$h(\ell)\in[q]$ for $\ell\in\N$ such that, for every fixed $\ell$,
\[
 h_\nu(x_{\nu,\ell})=h(\ell)
\]
eventually and
\[
 \rho_{i,\nu}(x_{\nu,\ell})\longrightarrow
 \rho_i(\ell):=\mu_i(A_i(\ell))
 \quad(i\in[q]).
\]
The limiting system satisfies \ref{it:N1} and \ref{it:N2}, and
$h(\ell)$ is a minimizing coordinate for $\rho_i(\ell)$.
\end{lemma}

\begin{proof}
For fixed $i$ and $L$, and for all sufficiently large $\nu$ with
$t_\nu\ge L$, define the incidence-pattern law on $\{0,1\}^L$ by
\[
 p_{i,\nu}^{(L)}(\sigma)
 =\mu_{i,\nu}\!\left(
 \left\{\omega:
 \ind_{A_{i,\nu}(x_{\nu,\ell})}(\omega)=\sigma_\ell
 \text{ for }\ell\in[L]\right\}\right).
\]
Each law lies in a finite-dimensional compact simplex.  A diagonal
subsequence makes every $p_{i,\nu}^{(L)}$ converge and makes every fixed
actual home stabilize.  The limiting laws are projectively consistent
because each finite law is the marginal of every larger one before taking
limits.  By the Kolmogorov extension theorem, for every $i$ they define a
Borel probability measure $\mu_i$ on $\Omega_i=\{0,1\}^{\N}$.  Let
\[
 A_i(\ell)=\{\omega\in\Omega_i:\omega_\ell=1\}.
\]
Cylinder probabilities converge, so in particular
$\rho_{i,\nu}(x_{\nu,\ell})\to\rho_i(\ell)$ for every fixed $i,\ell$.
Taking limits in \ref{it:A1} gives \ref{it:N1}.  If $S\subseteq\N$ has
order $a$, every common-intersection probability for the corresponding
fixed labels is a cylinder probability; the minimum over $[q]$ is
continuous.  Thus \ref{it:A2} gives \ref{it:N2}.  Finally, the stabilized
home is minimizing because
\[
 \rho_{h(\ell),\nu}(x_{\nu,\ell})=e_\nu(x_{\nu,\ell})
\]
and all coordinates converge.
\end{proof}

\begin{lemma}[Ordered impurity]\label{lem:orderedimpurity}
In the setting of \Cref{lem:fdlimit}, put
\[
 e(\ell)=\min_i\rho_i(\ell).
\]
Then
\[
 e_\nu(x_{\nu,\ell})\longrightarrow e(\ell)
 \quad\text{for each fixed }\ell,
\]
\[
 \sum_{\ell=1}^{\infty}e(\ell)<\infty,
\]
and
\begin{equation}\label{eq:movingtail}
 1\le\ell_\nu\le t_\nu,\quad \ell_\nu\to\infty
 \quad\Longrightarrow\quad
 e_\nu(x_{\nu,\ell_\nu})\to0.
\end{equation}
\end{lemma}

\begin{proof}
Fixed-label convergence follows from convergence of all $q$ coordinates and
continuity of the minimum.  The limiting system satisfies \ref{it:N1} and
\ref{it:N2}, so the product-space and Tonelli argument leading to
\Cref{eq:esummable} applies and gives the summability assertion.

If \Cref{eq:movingtail} failed, then along a subsequence
$e_\nu(x_{\nu,\ell_\nu})\ge\delta$ for some $\delta>0$.  For every fixed
$L$ and all large $\nu$, the ordering and $\ell_\nu\ge L$ would imply
$e_\nu(x_{\nu,L})\ge\delta$.  Hence $e(L)\ge\delta$ for every fixed $L$,
contradicting summability.
\end{proof}

\begin{proposition}[Escaping-vertex principle]\label{prop:escapingvertex}
In the finite systems above, let
$y_\nu=x_{\nu,\ell_\nu}$ with $\ell_\nu\to\infty$.  If the actual home of
$y_\nu$ is a fixed coordinate $i$ for all sufficiently large $\nu$, then,
for every $j\ne i$,
\[
 \rho_{j,\nu}(y_\nu)=1-o(1).
\]
\end{proposition}

\begin{proof}
By \Cref{eq:movingtail},
$\rho_{i,\nu}(y_\nu)=e_\nu(y_\nu)=o(1)$.  Hypothesis \ref{it:A1} gives
\[
 \sum_{j\ne i}\bigl(1-\rho_{j,\nu}(y_\nu)\bigr)
 \le e_\nu(y_\nu)+\varepsilon_\nu=o(1).
\]
Every summand is nonnegative, which proves the assertion.
\end{proof}

Before giving the formal definition, we explain the shadow construction.  A
bounded-rank finite obstruction may move partly or entirely out of every fixed
initial label set.  Positions whose labels stabilize become the visible shadow.
Positions whose labels diverge are safe to discard only when their actual home is
the same as the obstruction color; the moving-tail lemma then says that they are
asymptotically present in every other coordinate, so the finite blocker survives
on the stabilized part.

For each $i\in[q]$, define a shadow hypergraph $\cD_i$ on $\N$ as follows.
A finite set $E\subseteq\N$ is an edge if there are indices
$\nu_k\uparrow\infty$, an integer $r\le r_*$, and edges
$F_k\in\cC_{i,\nu_k}$ of order $r$ such that, on writing their labels as
\[
 \ell_{k,1}<\cdots<\ell_{k,r},
\]
for every position $p\in[r]$ the sequence $(\ell_{k,p})_k$ is either
eventually constant or tends to infinity.  The set of eventual constant
values must be exactly $E$, and whenever $\ell_{k,p}\to\infty$, the vertex
with that label must have actual home $i$ for all sufficiently large $k$.
The rank bound and repeated subsequence extraction make the
constant-or-divergent alternative available at every position; the
actual-home condition is an additional requirement for membership in
$\cD_i$.  For this definition alone, $E=\varnothing$ is provisionally
allowed; the next lemma shows that it cannot occur.
The collection of finite subsets of $\N$ is countable, and hence each shadow
hypergraph is countable.

\begin{lemma}[Shadow blockers]\label{lem:shadowblockers}
Every shadow hypergraph $\cD_i$ has rank at most $r_*$.  Every
$E\in\cD_i$ is nonempty and has a coordinate $j\ne i$ such that
\[
 \mu_j\left(\bigcap_{\ell\in E}A_j(\ell)\right)=0.
\]
\end{lemma}

\begin{proof}
The rank assertion is immediate.  Let $E\in\cD_i$ be represented by
$(\nu_k,F_k)$ as in the definition, and relabel this subsequence by $\nu$.
Choose a blocker $j_\nu\ne i$ from \ref{it:A3} and pass to a
subsequence on which $j_\nu=j$.  Write
\[
 F_\nu=E_\nu\mathbin{\dot\cup}Q_\nu,
\]
where $E_\nu$ comprises the eventually constant positions and $Q_\nu$ the
divergent positions.  For a divergent position $p$, let $y_{\nu,p}$ be its
vertex in $F_\nu$.  By \Cref{prop:escapingvertex},
$1-\rho_{j,\nu}(y_{\nu,p})=o(1)$.  The rank bound is used
precisely here: there are at most $r_*$ divergent positions, so the sum of
their $o(1)$ deficits is still $o(1)$.  Thus
\[
 \sum_{y\in Q_\nu}(1-\rho_{j,\nu}(y))=o(1).
\]
The inclusion
\[
 \bigcap_{x\in E_\nu}A_{j,\nu}(x)
 \subseteq
 \left(\bigcap_{x\in F_\nu}A_{j,\nu}(x)\right)
 \cup
 \bigcup_{y\in Q_\nu}
       (\Omega_{j,\nu}\setminus A_{j,\nu}(y))
\]
therefore gives
\[
 \mu_{j,\nu}\left(\bigcap_{x\in E_\nu}A_{j,\nu}(x)\right)
 \le\varepsilon_\nu+o(1).
\]
The stabilized labels are eventually exactly $E$, so the left side
converges to the corresponding cylinder probability.  Hence
\[
 \mu_j\left(\bigcap_{\ell\in E}A_j(\ell)\right)=0.
\]
If $E=\varnothing$, the left side is $\mu_j(\Omega_j)=1$, a contradiction.
Thus $E$ is nonempty.
\end{proof}

\begin{lemma}[Finite transfer]\label{lem:finitetransfer}
Suppose the limiting event system and the shadow hypergraphs admit a set
$Z\subseteq\N$, $|Z|\le a-1$, and a coloring
\[
 \phi:\N\setminus Z\to[q]
\]
with no $\cD_i$-edge monochromatic in color $i$, such that
$\phi(\ell)=h(\ell)$ for all sufficiently large $\ell$.  Then, for all
sufficiently large $\nu$, the finite counterexample sequence is actually
colorable after deleting at most $a-1$ vertices.
\end{lemma}

\begin{proof}
Choose $L$ so that $Z\subseteq[L]$ and
$\phi(\ell)=h(\ell)$ for every $\ell>L$.  For large $\nu$, delete the
vertices whose labels lie in $Z$.  Color every remaining label
$\ell\le L$ by $\phi(\ell)$ and every label $\ell>L$ by its actual home
$h_\nu(x_{\nu,\ell})$.

Suppose that this fails along an infinite subsequence.  Passing to a
subsequence, there are a fixed color $i$, a fixed edge order $r\le r_*$,
and edges $F_\nu\in\cC_{i,\nu}$ that are monochromatic in color $i$.  List
the labels increasingly and pass again so that every position is eventually
constant or tends to infinity.  Let $E$ be the stabilized label set.
Because the edge lies in the colored ground set, $E\cap Z=\varnothing$.
There are three cases for positions.  A stabilized label below or equal to
$L$ is colored by the fixed rule $\phi$ and hence has color $i$.  A
stabilized label above $L$ is eventually colored by its actual home; since
actual homes stabilize and $\phi=h$ above $L$, it also has $\phi$-color $i$.
A divergent label is eventually above $L$ and is colored by its actual home,
so its actual home is $i$.  Therefore $E\in\cD_i$ and every element of $E$
has color $i$.  By \Cref{lem:shadowblockers}, $E\ne\varnothing$,
contradicting the choice of $\phi$.
\end{proof}

\begin{proof}[Proof of \Cref{thm:compactness}]
Suppose no $\varepsilon_0$ works.  Then for each $\nu$ there is a finite
counterexample satisfying \ref{it:A1}--\ref{it:A3} with
$\varepsilon_\nu\le1/\nu$.  Pass first to a subsequence on which either
$|X_\nu|$ is bounded or $|X_\nu|\to\infty$; such a subsequence exists for
every sequence of nonnegative integers.  In the bounded case, pass again
on which $|X_\nu|=t$ and label every ground set by $[t]$.  There are only
finitely many rank-at-most-$r_*$ hypergraphs on $[t]$, so all
$\cC_{i,\nu}$ may be made constant.  Pass again so that every
incidence-pattern law on $\{0,1\}^t$ converges.  The limiting finite event
system satisfies the exact hypotheses \ref{it:N1}--\ref{it:N3}, because all
relevant common-intersection probabilities converge and the minima range
over finitely many coordinates.  Apply \Cref{thm:exactrounding}.  Its
partition is independent in the stabilized hypergraphs and therefore works
for every sufficiently large finite system, a contradiction.

We are therefore in the alternative $|X_\nu|\to\infty$.  Choose actual homes, order the
labels by nonincreasing impurity, and apply
\Cref{lem:fdlimit,lem:orderedimpurity}.  Construct the shadow hypergraphs.
By \Cref{lem:shadowblockers}, the limiting event system and shadows satisfy
all hypotheses of \Cref{thm:exactrounding}.  That theorem supplies $Z$ and
$\phi$ with the home property required in \Cref{lem:finitetransfer}.  The
finite transfer lemma then gives the desired finite colorings for all
sufficiently large $\nu$, contradicting the assumed counterexamples.
\end{proof}

\section{Stable reservoirs in a hypothetical counterexample}\label{sec:reservoirs}

\noindent\textit{Proof of \Cref{thm:main}.}
Put
\[
 a=m_1,
 \qquad b=m_2,
 \qquad q=k-1,
 \qquad H=K_{a,b},
 \qquad F=K_{m_1,\ldots,m_k}.
\]
For $k=2$, we have $F=H$, and the right-hand side is
$R(T,H)-1+a\ge R(T,H)$, so the assertion is immediate.  Henceforth assume
$q\ge2$.

Suppose the theorem is false.  By interchanging the two colors in the
hypothetical counterexamples if necessary, there is a sequence $n\to\infty$,
a sequence of $n$-vertex trees $T_n$, and red--blue colorings of
\[
 K_N,
 \qquad
 N=q(r-1)+a,
 \qquad
 r=R(T_n,H),
\]
containing neither a red copy of $F$ nor a blue copy of $T_n$.  We suppress
the sequence index.  To keep every later uniformity claim explicit, define
\[
 \sigma_n=\max\left\{
  \left|\frac rn-1\right|,
  \left|\frac Nn-q\right|
 \right\}.
\]
By \Cref{prop:efrs} and $N=q(r-1)+a$, we have
\begin{equation}\label{eq:rnasymp}
 \sigma_n\longrightarrow0,
 \qquad
 \left|\frac rn-1\right|\le\sigma_n,
 \qquad
 \left|\frac Nn-q\right|\le\sigma_n.
\end{equation}
By \Cref{cor:nearTuran}, the named deficit
\[
 \theta_n=\left(\frac{t_q(N)-e(\Gr)}{N^2}\right)_+
\]
satisfies $\theta_n\to0$, and
\begin{equation}\label{eq:nearTurancounter}
 e(\Gr)\ge t_q(N)-\theta_nN^2.
\end{equation}

\begin{lemma}[Clean reservoir decomposition]\label{lem:reservoirs}
Along the fixed counterexample sequence, there is a sequence
$\lambda_n\to0$, and there are pairwise disjoint sets $W_1,\ldots,W_q$ and
\[
 X=V(K_N)\setminus\bigcup_{i=1}^qW_i
\]
such that
\begin{align}
 (1-\lambda_n)n\le |W_i|&\le(1+\lambda_n)n
       &&(i\in[q]),\label{eq:Wsize}\\
 \delta(\Gb[W_i])&\ge |W_i|-\lambda_n n
       &&(i\in[q]),\label{eq:Wblue}\\
 d_{\Gb}(w,W_j)&\le\lambda_n n
       &&(w\in W_i,\ i\ne j),\label{eq:crossblue}\\
 \Gr[W_i]&\text{ is }H\text{-free},\label{eq:WHfree}\\
 |X|&\le\lambda_n n.\label{eq:Xsmall}
\end{align}
\end{lemma}

\begin{proof}
The red graph $\Gr$ is $F$-free and $\chi(F)=q+1$.  By
\Cref{prop:stability,eq:nearTurancounter}, there is a partition
\[
V(K_N)=V_1\mathbin{\dot\cup}\cdots\mathbin{\dot\cup}V_q
\]
for which the internal red-edge count divided by $N^2$ tends to zero.
Choose such partitions and set
\[
 \upsilon_n:=\frac1{N^2}\sum_{i=1}^q e_{\Gr}(V_i)\longrightarrow0.
\]
Let $P=\sum_{i<j}|V_i||V_j|$.  The red crossing-edge count is at least
$t_q(N)-(\theta_n+\upsilon_n)N^2$, while $P\le t_q(N)$.  Hence
\[
 0\le t_q(N)-P\le(\theta_n+\upsilon_n)N^2.
\]
The identity
\[
 t_q(N)-P
 =\frac12\sum_{i=1}^q\left(|V_i|-\frac Nq\right)^2+O_q(1)
\]
therefore gives
\[
 \max_{i\in[q]}\left|\frac{|V_i|}{n}-1\right|\longrightarrow0.
\]
It also follows that the normalized total number of blue edges crossing
between the parts tends to zero.

Write
\[
 \beta_n=\frac{1}{n^2}\sum_{i<j}e_{\Gb}(V_i,V_j)\longrightarrow0,
 \qquad
 \gamma_n=\sqrt{\beta_n}+n^{-1/2}.
\]
Delete every vertex having more than $\gamma_n n$ blue neighbors in some
other part, and call the remaining sets $U_i$.  By the union bound and the
handshake lemma, the number deleted is at most
\[
 \frac{2\sum_{i<j}e_{\Gb}(V_i,V_j)}{\gamma_n n}
 =\frac{2\beta_n}{\gamma_n}n,
\]
where $2\beta_n/\gamma_n\to0$.
Thus
\begin{equation}\label{eq:Ucross}
 d_{\Gb}(u,U_j)\le\gamma_n n
 \qquad(u\in U_i,\ i\ne j).
\end{equation}

We claim that $\Gr[U_i]$ is $H$-free.  If a red copy of $H$ lay in $U_i$,
use it for the first two classes of $F$ and greedily choose classes of
orders $m_3,\ldots,m_k$, one from each of the other reservoirs.  Copies are
not required to be induced, so no edge restrictions are imposed within an
individual multipartite class.  At every stage at most $C_F=|F|$ vertices
have already been chosen, and the next reservoir contains at least
\[
 |U_j|-C_F\gamma_n n
\]
vertices red-adjacent to all previously selected vertices.  Divided by $n$,
this quantity tends to $1$, so it exceeds the fixed required class size
$m_s$ for large $n$.  Hence the greedy construction succeeds and gives a
red $F$, a contradiction.

By the K\H{o}v\'ari--S\'os--Tur\'an theorem,
$e_{\Gr}(U_i)/n^2\to0$ uniformly in $i$.  Put
\[
 \alpha_n=\frac{1}{n^2}\sum_{i=1}^q e_{\Gr}(U_i)\longrightarrow0,
 \qquad
 \zeta_n=\sqrt{\alpha_n}+n^{-1/2}.
\]
Delete from every $U_i$ each vertex whose red degree inside $U_i$ exceeds
$\zeta_n n$, and call the remaining set $W_i$.  The total number deleted
in this second cleaning is at most
\[
 \frac{2\sum_i e_{\Gr}(U_i)}{\zeta_n n}
 =\frac{2\alpha_n}{\zeta_n}n,
\]
where $2\alpha_n/\zeta_n\to0$.  Consequently all four normalized error
quantities in the definition below tend to zero: indeed, deleting vertices
preserves \Cref{eq:Ucross}, and for $w\in W_i$,
\[
 d_{\Gb}(w,W_i)\ge |W_i|-1-\zeta_n n.
\]
After discarding finitely many terms of the counterexample sequence, every
$U_i$ and every $W_i$ is nonempty, because each has order $n-o(n)$.
Define
\[
 \lambda_n=\max\left\{
  \max_i\left|\frac{|W_i|}{n}-1\right|,
  \max_i\frac{|W_i|-\delta(\Gb[W_i])}{n},
  \max_{\substack{i\ne j\\w\in W_i}}\frac{d_{\Gb}(w,W_j)}n,
  \frac{|X|}{n}
 \right\},
\]
where the third maximum is $0$ if its indexing set is empty.  Then
$\lambda_n\to0$, and \Cref{eq:Wsize,eq:Xsmall} hold.  Deleting vertices cannot
create a red $H$, so \Cref{eq:WHfree} holds; it also preserves the
cross-blue bound, giving \Cref{eq:crossblue}.  Finally, for $w\in W_i$,
the displayed internal-degree bound proves \Cref{eq:Wblue}.
\end{proof}

For $x\in X$, define the normalized red profile
\[
 \rho_i(x)=\frac{d_{\Gr}(x,W_i)}{|W_i|}.
\]
Because $x\notin W_i$ and every edge $xw$ has exactly one color,
\begin{equation}\label{eq:profilecomplement}
 \rho_i(x)=1-\frac{d_{\Gb}(x,W_i)}{|W_i|}.
\end{equation}
Define
\[
 \pi_n=\sup_{x\in X}
 \left(\sum_{i=1}^q\frac{d_{\Gb}(x,W_i)}{|W_i|}-1\right)_+,
\]
with value $0$ when $X=\varnothing$.  We claim that $\pi_n\to0$.  Otherwise
there are a fixed $\omega'>0$, an infinite subsequence, and vertices $x\in X$
on that subsequence such that
\[
 \sum_{i=1}^q\frac{d_{\Gb}(x,W_i)}{|W_i|}\ge1+\omega'.
\]
By \Cref{eq:Wsize}, once $\lambda_n\le\omega'/(2+2\omega')$ this implies
\[
 \sum_{i=1}^q d_{\Gb}(x,W_i)
 \ge(1-\lambda_n)(1+\omega')n
 \ge(1+\omega'/2)n.
\]
Apply \Cref{lem:profile} with $\omega'/2$.  Its hypotheses follow from
\Cref{eq:Wsize,eq:Wblue} once $\lambda_n$ is at most the resulting $\gamma$
and $n$ is at least the resulting $n_0$.  This produces a blue copy of
$T_n$, a contradiction.  Hence $\pi_n\to0$.  Summing
\Cref{eq:profilecomplement} gives, for every $x\in X$,
\begin{equation}\label{eq:redprofile}
 \sum_{i=1}^q\rho_i(x)\ge q-1-\pi_n.
\end{equation}

\section{Obstruction hypergraphs}\label{sec:blockers}

Take $\Omega_i=W_i$ with the uniform probability measure $\mu_i$, and for
$x\in X$ put
\[
 A_i(x)=N_{\Gr}(x,W_i)\subseteq\Omega_i.
\]
Thus $\mu_i(A_i(x))=\rho_i(x)$.

For each $i\in[q]$, define a hypergraph $\cC_i$ on $X$.  A nonempty set
$E\subseteq X$ is an edge of $\cC_i$ if it is inclusion-minimal subject to
\begin{equation}\label{eq:obstructiondef}
 \Gr[W_i\cup E]\supseteq H.
\end{equation}
Since $\Gr[W_i]$ is $H$-free, every such edge is nonempty.  By minimality,
every red copy of $H$ witnessing \Cref{eq:obstructiondef} uses every vertex of
$E$: if a witnessing copy omitted $x\in E$, it would already lie in
$W_i\cup(E\setminus\{x\})$, contradicting inclusion-minimality.  Hence
\begin{equation}\label{eq:rank}
 |E|\le |H|=a+b.
\end{equation}

\begin{lemma}[$a$-set separation]\label{lem:asetblock}
Along the fixed counterexample sequence,
\[
 \sup_{S\in\binom Xa}\min_{i\in[q]}
 \mu_i\left(\bigcap_{x\in S}A_i(x)\right)\longrightarrow0,
\]
where the supremum is $0$ if $\binom Xa=\varnothing$.
\end{lemma}

\begin{proof}
Otherwise, along a subsequence there are a fixed $\delta>0$ and $a$-sets
$S$ whose common red-neighborhood density is at least $\delta$ in every
reservoir.  Use $S$ as the first class of $F$, choose the $b$-vertex second
class from its common red neighborhood in $W_1$, and then choose classes of
orders $m_3,\ldots,m_k$ successively from $W_2,\ldots,W_q$.  At a later
stage, at most $C_F=|F|$ reservoir vertices have already been chosen.  The
next reservoir therefore has at least
\[
 \delta|W_j|-C_F\lambda_n n.
\]
These candidates are red-adjacent to all previously selected vertices and to
every vertex of $S$.  Since $|W_j|/n\to1$ and $\lambda_n\to0$, this exceeds
the fixed required class size for large $n$.
The greedy selection succeeds and gives a red $F$.
\end{proof}

\begin{lemma}[Obstruction blocking]\label{lem:obstructionblock}
Along the fixed counterexample sequence,
\[
 \max_{i\in[q]}\sup_{E\in\cC_i}\min_{j\ne i}
 \mu_j\left(\bigcap_{x\in E}A_j(x)\right)\longrightarrow0,
\]
where the supremum is $0$ when $\cC_i=\varnothing$.
\end{lemma}

\begin{proof}
Otherwise, along a subsequence there are a fixed $\delta>0$, indices $i$,
and minimal obstructions $E\in\cC_i$ whose common red-neighborhood density is
at least
$\delta$ in every $W_j$ with $j\ne i$.  Let $Q\cong H$ be a red witness in
$W_i\cup E$.  By inclusion-minimality of $E$, the witness uses every vertex
of $E$: if it omitted some $x\in E$, the same copy would lie in
$W_i\cup(E\setminus\{x\})$, contradicting minimality.  Hence
$Q\cap E=E$.  Every vertex of $Q\cap W_i$ has at most $\lambda_n n$ blue neighbors in
every other reservoir by \Cref{eq:crossblue}.  Thus all vertices of $Q$
have a common red neighborhood of order at least
$\delta|W_j|-|H|\lambda_n n$ in each $W_j$, $j\ne i$.  At each greedy
extension step the available common red candidate set has size at least
$\delta |W_j|-C_F\lambda_n n$.  Divided by $n$, this is at least
$\delta(1-\lambda_n)-C_F\lambda_n\to\delta$, so it exceeds the next fixed
class size.  Greedily choosing the remaining
$k-2$ classes of $F$ gives a red $F$, a contradiction.
\end{proof}

For the fixed counterexample sequence, define
\[
 \eta_n^{(1)}=
 \sup_{x\in X}\left(q-1-\sum_{i=1}^q\rho_i(x)\right)_+,
\]
\[
 \eta_n^{(2)}=
 \sup_{S\in\binom Xa}
 \min_{i\in[q]}\mu_i\left(\bigcap_{x\in S}A_i(x)\right),
\]
and
\[
 \eta_n^{(3)}=
 \max_{i\in[q]}\sup_{E\in\cC_i}
 \min_{j\ne i}\mu_j\left(\bigcap_{x\in E}A_j(x)\right),
\]
where a supremum over an empty family is $0$.  By
\Cref{eq:redprofile,lem:asetblock,lem:obstructionblock},
\[
 \eps_n:=\max\{\eta_n^{(1)},\eta_n^{(2)},\eta_n^{(3)}\}\longrightarrow0.
\]
Let $\varepsilon_0$ be the constant in \Cref{thm:compactness} for
$r_*=a+b$.  For all sufficiently large $n$ we have $\eps_n\le\varepsilon_0$.
The systems $(\Omega_i,\mu_i,A_i(x))$ and the obstruction hypergraphs $\cC_i$
therefore satisfy \Cref{thm:compactness} with rank bound $r_*=a+b$.

\section{Final capacity contradiction}\label{sec:mainproof}

Apply \Cref{thm:compactness}.  For all sufficiently large members of the
counterexample sequence, there are $Z\subseteq X$, $|Z|\le a-1$, and a
partition
\[
 X\setminus Z=X_1\mathbin{\dot\cup}\cdots\mathbin{\dot\cup}X_q
\]
such that $X_i$ is independent in $\cC_i$.

We claim that
\begin{equation}\label{eq:classHfree}
 \Gr[W_i\cup X_i]\text{ is }H\text{-free}
 \qquad(i\in[q]).
\end{equation}
Indeed, a red copy of $H$ in $W_i\cup X_i$ would use a nonempty set of
vertices from $X_i$, because $\Gr[W_i]$ is $H$-free.  Among the vertices of
the copy lying in $X_i$, choose an inclusion-minimal set $E$ for which
$W_i\cup E$ contains
a red $H$.  Then $E\in\cC_i$ and $E\subseteq X_i$, contradicting
independence.

Since
\[
 r=R(T_n,H)=R(H,T_n),
\]
every complete graph on $r$ vertices contains a red $H$ or a blue $T_n$.
By \Cref{eq:classHfree} and the global absence of a blue $T_n$,
\[
 |W_i\cup X_i|\le r-1
 \qquad(i\in[q]).
\]
The $q$ sets cover all vertices except $Z$, so
\[
 N-|Z|=\sum_{i=1}^q|W_i\cup X_i|\le q(r-1).
\]
On the other hand,
\[
 N-|Z|\ge q(r-1)+a-(a-1)=q(r-1)+1,
\]
a contradiction.  This proves \Cref{thm:main}.\hfill\qedsymbol

\section{Concluding remarks}

The proof uses the complement restriction in the off-Tur\'an theorem before
any arbitrary-tree embedding is attempted.  This reduces the relevant
reduced graph to one with small independence number; a head edge, an
almost-spanning matching, balanced portions of the matched clusters, and the
whole-edge/stateful embedding argument then suffice.
The remaining Ramsey argument is separated into a stability stage and the
null-blocker compactness theorem.

The compactness theorem is formulated independently of Ramsey theory.  Its
shadow construction is needed because bounded-rank obstructions may retain
no fixed labels along a limiting sequence.  Ordering vertices by their
minimum event measure forces every escaping vertex to be asymptotically
complete in every nonhome coordinate and prevents such an obstruction from
vanishing.

\section*{Formal verification}

The complete, unconditional resolution of Erd\H{o}s Problem~550 has been fully
formalised and machine-verified in Lean, using toolchain v4.28.0 and Mathlib
v4.28.0.  The exported
declaration \texttt{Erdos550.erdos\_550} formalises exactly the quantified
content of \Cref{thm:main}: it covers every fixed number of parts, every
nondecreasing sequence of positive part sizes, and every sufficiently large
tree, with no unproved mathematical hypothesis.  Its dependency closure
consists of 134 Lean files and contains no \texttt{sorry} or user-declared
axiom.
The axiom audit in \texttt{MainAudit.lean} reports only the standard Mathlib
foundation \texttt{propext}, \texttt{Classical.choice}, and
\texttt{Quot.sound}.  The formal proof includes the direct off--Tur\'an
embedding theorem, exact regularity and reduced-graph accounting,
parity-refined tree data, whole-edge allocation, the stateful packed
regular-pair embedding, stability and reservoir reduction, bounded-rank
obstruction compactness, and the final Ramsey-capacity argument.  Complete
sources, the declaration-level proof map, and reproduction instructions are
included in the accompanying repository~\cite{Erdos550Lean}.

\section*{Use of artificial intelligence tools}

Large language models, primarily OpenAI's ChatGPT and Codex, were used
extensively throughout the research.  The author originated the ideas and
research directions, while the models contributed substantially to the
technical development of the work: they were used to explore and further
develop the ideas, produce technical lemmas and calculations, and generate
and debug code.  The author critically evaluated outputs from different model
instances, selected and synthesised promising elements, and discarded or
corrected unsuccessful or flawed suggestions.  The Lean formalisation was
produced with Harmonic's Aristotle system and ChatGPT's Codex, directed and
audited by the author.  All strategic decisions were made by the author, who
takes full responsibility for the mathematical correctness of the final
results.

\end{document}